\documentclass[a4paper,draft,reqno,12pt]{amsart}
\usepackage[utf8]{inputenc}
\usepackage[T1,T2A]{fontenc}
\usepackage[english]{babel}
\usepackage{amsmath}
\usepackage{amssymb}
\usepackage{amscd}
\usepackage{amsthm}
\usepackage{euscript}
\usepackage[all]{xy}
\newtheorem{proposition}{Proposition}
\newtheorem{conjecture}{Conjecture}
\newtheorem{problem}{Problem}

\newtheorem{lemma}{Lemma}
\newtheorem{theorem}{Theorem}

\newtheorem{corollary}{Corollary}
\theoremstyle{definition}
\newtheorem{definition}{Definition}
\newtheorem{example}{Example}
\theoremstyle{remark}
\newtheorem {remark}{Remark}

\DeclareMathOperator{\Aut}{Aut}

\DeclareMathOperator{\Der}{Der}

\DeclareMathOperator{\ad}{ad}
\DeclareMathOperator{\Ker}{Ker}

\def\FF{{\mathbb F}}

\def\QQ{{\mathbb Q}}

\def\AAA{\mathcal{A}}
\def\BBB{\mathcal{B}}

\def\RRR{\mathfrak{R}}
\def\SSS{\mathfrak{S}}
\def\ggg{\mathfrak{g}}
\def\sss{\mathfrak{s}}
\def\fff{\mathfrak{f}}

\def\hhh{\mathfrak{h}}
\def\ppp{\mathfrak{p}}
\def\bbb{\mathfrak{b}}
\def\nnn{\mathfrak{n}}
\def\kkk{\mathfrak{k}}
\def\lll{\mathfrak{l}}
\def\rrr{\mathfrak{r}}
\def\aaa{\mathfrak{a}}
\def\gl{\mathfrak{gl}}
\def\sl{\mathfrak{sl}}
\def\so{\mathfrak{so}}
\def\sp{\mathfrak{sp}}

\renewcommand{\phi}{\varphi}
\renewcommand{\ge}{\geqslant}
\renewcommand{\le}{\leqslant}

\begin{document}
\date{}
\title[Uniqueness of addition in Lie algebras]{Uniqueness of addition in Lie algebras revisited}
\author{Ivan Arzhantsev}
\address{Faculty of Computer Science, HSE University, Pokrovsky Boulevard 11, Moscow, 109028 Russia}
\email{arjantsev@hse.ru}
\thanks{The research was done within the framework of the HSE Fundamental Research Program}

\maketitle
\begin{abstract}
We obtain new and improve old results on uniqueness of addition in Lie rings and Lie algebras. A Lie ring $\RRR$ is called a unique addition ring, or a UA-Lie ring, if any commutator-preserving bijection from $\RRR$ to an arbitrary Lie ring is additive. We describe wide classes of Lie rings that are not UA-Lie ring. In the other direction, it is known that if a finite-dimensional Lie algebra $\ggg$ contains two elements whose centralizers have trivial intersection, then $\ggg$ is a UA-Lie ring. We use this result to characterize UA-Lie rings among seaweed Lie algebras. The paper includes many open problems and questions. 
\end{abstract} 

%%%%%%%%%%%%%%%%%%%

\section{Introduction}
\label{sec1}

Let us consider two algebraic structures $\AAA$ and $\BBB$ of the same type. A homomorphism between $\AAA$ and $\BBB$ is a map $\alpha\colon\AAA\to\BBB$ that preserves all operations of the corresponding structure. It is useful to know that if $\alpha$ preserves some of the operations, then all other operations are preserved automatically. And it does happen sometimes. 

We consider the case when the algebraic structure is a ring structure and the map $\alpha$ is bijective. There is no chance that any bijection that preserves addition will automatically preserve multiplication. But there is a hope that for some rings $\RRR$ all bijections from $\RRR$ to an arbitrary ring $\SSS$ preserving multiplication are additive maps. In this case one calls $\RRR$ a unique addition ring, or a UA-ring for short. The term is explained by the equivalent reformulation of the property of unique addition: for the operation of multiplication $(\RRR,\cdot)$  there is a unique binary operation $(\RRR,+)$ such that $(\RRR,+,\cdot)$ is a ring.

The study of associative rings with unique addition is a well-developed area of ring theory with many meaningful results; see e.g. \cite{Jo,Mar,Mi,Ri,St}. In Section~\ref{rua}, we recall a proof of one of the main results in this field. Namely, the theorem by Rickart and Stephenson states that for any associative ring $A$ with unit and any $n\ge 2$, the ring of matrices $\text{Mat}_{n\times n}(A)$  is a UA-ring.

The main object of study in this paper are Lie rings and Lie algebras. In Section~\ref{lua} we give a definition of a Lie ring with a unique addition, or a UA-Lie ring. Also we discuss a weaker form of this notion that sometimes appears in the literature and causes certain misunderstandings. 

In 1992, Konstantin Beidar asked the author whether the Lie algebra $\mathfrak{sl_n}(K)$ is a UA-Lie ring.  Some notions introduced below grew out of trying to solve this problem.

First results on the uniqueness of addition in Lie rings appeared more than 20 years ago. They have been published in several papers, some of which are now difficult to access. The results were not formulated in full generality, and recently there has been a desire to review and systematize the results obtained at that time, to find possible generalizations, as well as to give brief and uniform proofs. This desire is realized in this work. In addition, we obtain several new results. Besides the study of the uniqueness of addition by itself, this topic is a good motivation to develop the structural theory of Lie rings and Lie algebras over various fields.

In Section~\ref{nr} we collect results showing that certain Lie rings are not UA-Lie rings. Proposition~\ref{neg} claims that if a Lie ring $\RRR$ is not perfect, has a non-trivial center and is big enough, then $\RRR$ is not a UA-Lie ring. 

Starting from Section~\ref{cc} we come to positive results. The crucial point is that the class of UA-Lie rings is wide -- many important Lie algebras turn out to be Lie rings with unique addition, so this property is not exotic. 

Let us say that a Lie ring $\RRR$ satisfies the C-condition if there are two elements in $\RRR$ whose centralizers have trivial intersection. In Theorem~\ref{TCC} following~\cite{A4} we prove
that any finite-dimensional Lie algebra over an infinite field satisfying the C-condition is a UA-Lie ring. To be honest, to check the C-condition is the only known way to prove that a Lie algebra is a UA-Lie ring. At the same time,  in Section~\ref{sla}  we explain that the class of Lie algebras with the C-condition includes all semisimple Lie algebras and all their parabolic subalgebras (Proposition~\ref{PCC}). Moreover, the concept of Mal'tsev splitting (Section~\ref{ms}) allows to impose certain admissibility conditions on the radical of a finite dimensional Lie algebra under which the C-condition is met (Theorem~\ref{tar}). 

In Section~\ref{sw} we consider a popular class of regular subalgebras of semisimple Lie algebras called seaweed subalgebras. In Proposition~\ref{psw} we give an effective necessary and sufficient condition for a seaweed Lie algebra to be a UA-Lie ring. The results of this section did not apper in earlier publications. 

The aim of Section~\ref{uala} is to discuss a version of the unique addition property which takes into account the specifics of Lie algebras. We say that a Lie algebra $\ggg$ over a field $K$  
is a UA-Lie algebra if any bijection $\alpha\colon\ggg\to\sss$ to a Lie algebra $\sss$ over the same field $K$ is a semilinear map. 

In the last short section we consider commutator-preserving injections between Lie rings. Currently, we have no positive results in this direction, but we expect they will appear soon. In particular, we conjecture that the rigidity of the commutator structure of a semisimple Lie algebra will allow to prove that every injective commutator-preserving map from such an algebra is additive.

The text includes a number of open problems. They are designed to show what kind of questions in this area can be worked on now. The main problem -- find necessary and sufficient conditions for a finite-dimensional Lie algebra $\ggg$ to be a UA-Lie ring -- is still very far from being solved, and this area contains many interesting and exciting tasks.

\smallskip

The author is grateful to the organizers of the European Non-Associative Algebra Seminar for invitation to give a talk on May 15, 2023. The talk served as an incentive to return to the subject and to write this paper. 
Special thanks are due to the referee for careful reading the text and valuable comments.

%%%%%%%%%%%%%%%%%%%%%%%%%%%
\section{Associative rings with unique addition}
\label{rua}

We begin with the case of associative rings. 

\begin{definition}
\label{dd}
An associative ring $(R,+,\cdot)$ is called a {\it unique addition ring}, or a {\it UA-ring} for short, if for any ring $(S,\oplus,*)$ any isomorphism $\alpha\colon R\to S$ of the multiplicative semigroups $(R,\cdot)$ and $(S,*)$ is an isomorphism of rings, i.e. the condition $\alpha(a\cdot b)=\alpha(a)*\alpha(b)$ for all $a,b\in R$ implies $\alpha(a+b)=\alpha(a)\oplus\alpha(b)$ for all $a,b\in R$. 
\end{definition} 
The semigroups $(R,\cdot)$ and $(S,*)$ contain unique zeroes, so $\alpha(0_R)=0_S$. If $R$ is a ring with unit, then the semigroups $(R,\cdot)$ and $(S,*)$ contain unique units, and again $\alpha(1_R)=1_S$. 

\smallskip

Equivalently, an associative ring $(R,+,\cdot)$ is a UA-ring, if for the semigroup $(R,\cdot)$ there exists a unique binary operation $+$ on $R$ making $(R,+,\cdot)$ into a ring. Indeed, given two additions $+$ and $\oplus$ on $R$, it follows from Definition~\ref{dd} applied to the identity map from $(R,+,\cdot)$ to $(R,\oplus,\cdot)$ that $a+b=a\oplus b$ for all $a,b \in R$. Conversely, every isomorphism of semigroups $\alpha\colon R\to S$ defines a new addition on $R$, namely $a\boxplus b:=\alpha^{-1}(\alpha(a)\oplus\alpha(b)).$

\smallskip

Let us test this definition in the case of finite fields. It is straightforward to check that for $q\le 4$ the field $\FF_q$ is a UA-ring. We show that for $q>4$ it is not the case.

Let $q=p^n$. We have $\Aut(\mathbb{F}_q)=\langle \psi\rangle$, where $\psi$ is the Frobenius automorphism, and $|\Aut(\mathbb{F}_q)|=n$. At the same time, the number of automorphisms of the semigroup $(\mathbb{F}_q, \cdot)$ is the Euler function $\varphi(q-1)$. For $q>4$ we have $n+1$ numbers $1,p,p^2,\ldots,p^{n-1},p^n-2$
that are smaller than $q-1$ and coprime with $q-1$. So $\varphi(q-1)>n$ and we can find a non-additive automorphism of the semigroup $(\mathbb{F}_q, \cdot)$.

\smallskip 

We do not try to overview the theory of unique addition associative rings, but just give one important result that may be considered as the starting point of the story we are going to tell. 

\begin{theorem} [Rickart~\cite{Ri}, Stephenson~\cite{St}] 
\label{RS}
Let $A$ be an associative ring with unit and~${n\ge 2}$. Then the matrix ring $\text{Mat}_{n\times n}(A)$ is a UA-ring.
\end{theorem}  

For convenience of the reader, we provide a short proof of this result following~\cite{St}. 

\smallskip

{\it Step 1.}\ Let $\alpha\colon R\to S$ be an isomorphism of multiplicative semigroups and $R=B\times C$, where $B$ and $C$ are left ideals in $R$. Then $\alpha(b+c)=\alpha(b)\oplus\alpha(c)$ for all $b\in B, c\in C$.

\begin{proof} 
Let $1_R=b_1+c_1$ with $b_1\in B$, $c_1\in C$. We have $b_1=b_1(b_1+c_1)$, so $b_1^2=b_1$ and $b_1c_1=0$. Let $\alpha(b_1)\oplus\alpha(c_1)=\alpha(d)$. We have 
$$\alpha(b_1d)=\alpha(b_1)\alpha(d)=\alpha(b_1)(\alpha(b_1)\oplus\alpha(c_1))=\alpha(b_1^2)\oplus\alpha(b_1c_1)=\alpha(b_1), 
$$
so $b_1d=b_1$, $c_1d=c_1$, and $d=(b_1+c_1)d=b_1+c_1=1_R$. For $x=b+c$ we have $b=xb_1$, $c=xc_1$, and
$$
\alpha(b)\oplus\alpha(c)=\alpha(xb_1)\oplus\alpha(xc_1)=\alpha(x)(\alpha(b_1)\oplus\alpha(c_1))=\alpha(x)=\alpha(b+c).
$$
\end{proof}
 
{\it Step 2.}\ Let $s\colon B\to C$ be an $R$-homomorphism. Then $\alpha(s(b)+c)=\alpha(s(b))\oplus\alpha(c)$ for all $b\in B$, $c\in C$. 

\begin{proof}
We have $B\times C=(\text{id}+s)(B)\times C=R$ and applying Step~1 three times we obtain 
$$
\alpha(b)\oplus\alpha(s(b)+c)=\alpha(b+s(b)+c)=\alpha((\text{id}+s)(b)+c)=
$$
$$
=\alpha((\text{id}+s)(b))\oplus\alpha(c)=\alpha(b)\oplus\alpha(s(b))\oplus\alpha(c),
$$ 
so the claim. 
\end{proof}

{\it Step 3.}\ Denote by $E_{ij}$ the matrix units in $\text{Mat}_{n\times n}(A)$. With $R=\text{Mat}_{n\times n}(A)$ we have $R=RE_{11}\times\ldots\times RE_{nn}$, and the left ideals $RE_{ii}$ are pairwise isomorphic. Applying the results of Step~1 and Step~2 to this decomposition, we complete the proof of Theorem~\ref{RS}. 

\bigskip

Let us consider one more variant of definition of a unique addition ring.

\begin{definition}
\label{dwua}
An associative ring $(R,+,\cdot)$ is called a {\it ring with weak unique addition}, or a {\it wUA-ring}, if any bijection $\alpha\colon R\to R$ preserving multiplication preserves the addition as well. 
\end{definition} 

Clearly, any UA-ring is a wUA-ring. But Definition~\ref{dwua} is weaker than Definition~\ref{dd}. Indeed, take two rings with zero multiplication (or commutative Lie rings): 
$$
\mathfrak{R}_1=(\mathbb{Z}/2\mathbb{Z}\oplus\mathbb{Z}/2\mathbb{Z},+)\  \ \text{and} \ \ \mathfrak{R}_2=(\mathbb{Z}/4\mathbb{Z},+).
$$ 
Then any bijection on $\mathfrak{R_1}$ that preserves $0_1$ is additive, but any bijection from $\mathfrak{R_1}$ to $\mathfrak{R_2}$ that sends $0_1$ to $0_2$ is not additive. 
So $\mathfrak{R_1}$ is a wUA-ring, but not a UA-ring. 

\smallskip

It seems that Definition~\ref{dd} is more natural and common in the literature. 

%%%%%%%%%%%%%%%%%%%%%%%%%%%
\section{Lie rings with unique addition}
\label{lua}

Now we come to uniqueness of addition in Lie rings.

\begin{definition}
\label{deflie}
A Lie ring $(\RRR,+,[,])$ is called a {\it unique addition Lie ring}, or a {\it UA-Lie ring} for short, if any Lie ring $(\SSS,\oplus,\lceil,\rceil)$ any bijection $\alpha\colon \RRR\to\SSS$ preserving commutators is an isomorphism of Lie rings, i.e. the condition $\alpha([a,b])=\lceil\alpha(a),\alpha(b)\rceil$ for all $a,b\in\RRR$ implies $\alpha(a+b)=\alpha(a)\oplus\alpha(b)$ for all $a,b\in\RRR$. 
\end{definition} 
\smallskip
We have $\alpha(0)=\alpha([0,0])=\lceil\alpha(0),\alpha(0)\rceil=0.$ Clearly, bijections as in Definition~\ref{deflie} preserve any property of a Lie ring that can be formulated in terms of commutators only. For example, commutative, nilpotent, or solvable Lie rings go to Lie rings of the same type.  
\begin{remark}
If $\RRR=\RRR_1\oplus\RRR_2$ is a direct sum of Lie rings and $\RRR$ is a UA-Lie ring, it is easy to check that $\RRR_1$ and $\RRR_2$ are UA-Lie rings as well.
\end{remark}

Here we also have a weaker version of the unique addition property. 

\begin{definition}
\label{dada}
A Lie ring $\RRR$ is a {\it Lie ring with weak unique addition}, or a {\it wUA-Lie ring}, if any bijection $\alpha\colon\RRR\to\RRR$ preserving multiplication preserves the addition as well. 
\end{definition} 

Some results in~\cite{A1}-\cite{A4}, \cite{AT} and~\cite{Pon2} are obtained for this weaker property, and one of the aim of this paper is to return to more natural Definition~\ref{deflie}.

%%%%%%%%%%%%%%%%%%%%%%%%%%%%%%%%%%
\section{Negative results} 
\label{nr}

In this section we collect and unify examples of Lie rings that are not UA-Lie rings given in~\cite{A1,A2,A4,AT}. In order to show that a Lie ring $\RRR$ is not a UA-Lie ring it suffices to construct a non-additive bijection $\alpha\colon\RRR\to\RRR$ that preserves commutators. 

\begin{lemma}
\label{lemnon}
Let $(A,+)$ be a commutative group with $|A|>4$ and $a\ne b\in A\setminus\{0\}$. Then the bijection $\alpha\colon A\to A$, $\alpha(a)=b$, $\alpha(b)=a$ and $\alpha(x)=x$ for all $x\ne a,b$ is not additive.
\end{lemma} 

\begin{proof} 
Take $c\in A$, $c\ne 0,a,b,b-a$. Then $a+c=\alpha(a+c)\ne\alpha(a)+\alpha(c)=b+c$.
\end{proof}

Denote by $[\RRR,\RRR]$ the commutant of a Lie ring $\RRR$ and let
$$
Z(\RRR)=\{a\in\RRR \ | \ [a,b]=0 \ \forall b\in\RRR\}
$$
be the center of $\RRR$. 

\begin{proposition}
\label{neg}
Let $\RRR$ be a Lie ring with $|\RRR|>4$. Assume that $[\RRR,\RRR]\ne\RRR$, $Z(\RRR)\ne\{0\}$ and the condition $Z(\RRR)\cap [\RRR,\RRR]=\{0\}$ implies $|Z(\RRR)|>2$. 
Then $\RRR$ is not a UA-Lie ring. 
\end{proposition}

\begin{proof} 
{\it Case 1}.\ Let $\RRR$ be commutative. Then we take a bijection from Lemma~\ref{lemnon}. 

\smallskip

{\it Case 2}.\ Let $\RRR$ be not commutative and $Z(\RRR)\cap [\RRR,\RRR]=\{0\}$. Take $0\ne z_1,z_2\in Z(\RRR)$, $z_1\ne z_2$, and $a\in [\RRR,\RRR]\setminus\{0\}$. 
We let 
$$
\alpha(a+z_1)=a+z_2, \ \  \alpha(a+z_2)=a+z_1, \ \ \text{and} \ \ \alpha(x)=x \quad \forall x\ne a+z_1, a+z-2. 
$$
Since $a+z_1,a+z_2\notin[\RRR,\RRR]$, $[x,a+z_1]=[x,a+z_2]$ for all $x\in\RRR$ and $[a+z_1,a+z_2]=0$, the bijection $\alpha$ preserves commutators. On the other hand, it is not additive by Lemma~\ref{lemnon}.

\smallskip

{\it Case 3}.\ Let $Z(\RRR)\cap [\RRR,\RRR]\ne\{0\}$. Then we take $0\ne z\in Z(\RRR)\cap [\RRR,\RRR]$, $a\notin [\RRR,\RRR]$, and let
$$
\alpha(a)=a+z, \ \  \alpha(a+z)=a, \ \ \text{and} \ \ \alpha(x)=x \quad \forall x\ne a, a+z. 
$$
Again we have $a,a+z\notin[\RRR,\RRR]$, $[x,a]=[x,a+z]$ for all $x\in\RRR$ and $[a,a+z]=0$, so the bijection $\alpha$ preserves commutators. It is not additive by Lemma~\ref{lemnon}.
\end{proof} 

\begin{corollary}
Let $\RRR$ be a nilpotent Lie ring with $|\RRR|>4$. Then $\RRR$ is not a UA-Lie ring.
\end{corollary}

\begin{corollary}
Let $K$ be a field with at least three elements. Then the Lie algebra $\gl_n(K)$ is not a UA-Lie ring for any $n\ge 2$. 
\end{corollary}

\begin{remark}
Let $K$ be the field $\FF_2$. For $n=2$ we have $E=[E_{12},{E_{21}}]$. So for $n=2k$ Proposititon~\ref{neg} implies that $\gl_n(\FF_2)$ is not a UA-Lie ring. 
\end{remark}

\begin{problem}
Is the Lie algebra $\gl_{2k+1}(\FF_2)$, $k\ge 1$, a UA-Lie ring?
\end{problem}

In Section~\ref{ms} we will see many examples of solvable Lie rings that are UA-Lie rings. So the condition $[\RRR,\RRR]=\RRR$ is not necessary for a Lie ring to be a UA-Lie ring. 

\begin{problem}
Does there exist a Lie ring with a trivial center that is not a UA-Lie ring?
\end{problem} 

\begin{example}
\label{exy}
Let $K$ be a field of characteristic different from $2$ and $\ggg$ be the 6-dimensional Lie algebra over $K$ defined as the semidirect product
$\sl_2(K)\rightthreetimes\langle x,y,z\rangle$, where 
$$
[x,y]=z, \quad [x,z]=[y,z]=0, 
$$
and the homomorphism $\sl_2(K)\to\Der(\langle x,y,z\rangle)$ 
is given by the tautologial $\sl_2$-representation on $\langle x,y\rangle$ and the trivial $\sl_2$-representation on $\langle z\rangle$. We have
$[\ggg,\ggg]=\ggg$ and $Z(\ggg)=\langle z\rangle$. 
\end{example} 

\begin{problem}
Is the Lie algebra $\ggg$ from Example~\ref{exy} a UA-Lie ring? 
\end{problem}

%%%%%%%%%%%%%%%%%%%%%%%%%%%%%%%%%%%%%%%%%%%%%%
\section{The C-condition}
\label{cc}

In this section we introduce the main technical tool to prove that a Lie algebra $\ggg$ is a UA-Lie ring. 
Let us consider the centralizer of an element in a Lie ring: 
$$
C(a)=\{b\in\RRR \, | \, [a,b]=0\}. 
$$
The following key lemma was found by Alexander Titov in his school years, see~\cite[Lemma~2]{AT}. 

\begin{lemma} 
\label{TL}
Let $\alpha\colon\RRR\to\SSS$ be a commutator-preserving bijection of Lie rings. If  
for some $a,b\in\RRR$ we have $C(a)\cap C(b)=\{0\}$, then $\alpha(a+b)=\alpha(a)\oplus\alpha(b)$. 
\end{lemma}
\begin{proof}
Let $\alpha(a+b)=\alpha(a)\oplus\alpha(b)\oplus\alpha(c)$ for some $c\in\mathfrak{R}$. We have 
$$
\lceil\alpha(a),\alpha(b)\rceil=\alpha([a,b])=\alpha([a+b,b])=\lceil\alpha(a+b),\alpha(b)\rceil=
$$
$$
=\lceil\alpha(a)\oplus\alpha(b)\oplus\alpha(c),\alpha(b)\rceil=\lceil\alpha(a),\alpha(b)\rceil\oplus\lceil\alpha(c),\alpha(b)\rceil.
$$
Hence $0=\lceil\alpha(c),\alpha(b)\rceil=\alpha([c,b])$, and $c\in C(b)$. By the same arguments, we have $c\in C(a)$, and so $c=0$. 
\end{proof} 

\begin{definition}
A Lie ring $\mathfrak{R}$ satisfies the {\it C-condition} if there are elements $a,b\in \mathfrak{R}$ with $C(a)\cap C(b)=\{0\}$. 
\end{definition}

\begin{remark}
Any Lie ring with C-condition has trivial center. Conversely, any 2-generated Lie ring (Lie algebra) with trivial center satisfies the C-condition. In particular, every semisimple Lie algebra $\ggg$ over a field of characteristic zero has trivial center and is generated by two elements~\cite[Theorem~6]{Ku}, so $\ggg$ satisfies the C-condition. 
\end{remark}

The following theorem is obtained in \cite[Theorem~3.2]{A4}. Below we give a proof of this results with some improvements and corrections. 

\begin{theorem}
\label{TCC}
Any finite-dimensional Lie algebra over an infinite field satisfying the C-condition is a UA-Lie ring.
\end{theorem} 

\begin{proof} 
For a Lie algebra $\mathfrak{g}$ with C-condition we let
$$
C(\mathfrak{g})=\{a\in\mathfrak{g}\ | \  \exists \ b\in\mathfrak{g} \ \text{with} \ C(a)\cap C(b)=\{0\}\}.
$$

\begin{lemma}
For any $a\in C(\mathfrak{g})$ the set $U(a)=\{b\in\mathfrak{g} \mid C(a)\cap C(b)=\{0\} \}$ is open and dense in Zariski topology on $\mathfrak{g}$.
\end{lemma}
\begin{proof}
Consider the linear map
$$
\mathfrak{g}\to\mathfrak{g}\oplus\mathfrak{g}, \ \ \ \  x\mapsto ([x,a], [x,y]). 
$$
At the point $y=b$ the rank of this map is maximal. So the rank is maximal on a dense open subset $U(a)$ in $\mathfrak{g}$.
\end{proof}

Since the set $C(\ggg)$ is the union of subsets $U(a), a\in C(\ggg)$, the set $C(\ggg)$ is open and dense in $\ggg$ as well. 

\smallskip

Now let $\alpha\colon\ggg\to\SSS$ be a commutator-preserving bijection of Lie rings.

\bigskip

{\it Case 1.}\ Suppose $a, b\in C(\mathfrak{g})$. The intersection $U(b)\cap(U(a)-b)$ of Zariski open subsets is nonempty and we take an element $c$ from this intersection.
By Lemma~\ref{TL}, we obtain
$$
\alpha(a)\oplus\alpha(b)\oplus\alpha(c)=\alpha(a)\oplus\alpha(b+c)=\alpha(a+b+c)=\alpha(a+b)\oplus\alpha(c)\oplus\alpha(d),
$$
with some $d\in\ggg$, and $\alpha(a)\oplus\alpha(b)=\alpha(a+b)\oplus\alpha(d)$. As we have seen in the proof of Lemma~\ref{TL},
the equality $\alpha(a+b+c)=\alpha(a+b)\oplus\alpha(c)\oplus\alpha(d)$ implies that the elements $c$ and $d$ commute. 
Since $c$ runs through a dense open subset, we conclude that $d\in Z(\ggg)$. But the algebra $\ggg$
satisfies the C-condition, so its center is trivial and $d=0$. 

\bigskip

{\it Case 2.}\ Suppose $b\in C(\mathfrak{g})$. Taking an element $a_1$ from $C(\mathfrak{g})\cap(a-C(\mathfrak{g}))\cap(a+b-C(\mathfrak{g}))$
and letting $a_2=a-a_1$, we obtain $a_1, a_2\in C(\mathfrak{g})$ with $a_1+a_2=a$ and $a_2+b\in C(\mathfrak{g})$. Using Case~1, we get
$$
\alpha(a)\oplus\alpha(b)=\alpha(a_1)\oplus\alpha(a_2)\oplus\alpha(b)=\alpha(a_1)\oplus\alpha(a_2+b)=\alpha(a_1+a_2+b)=\alpha(a+b).
$$

\bigskip

{\it Case 3.}\  Let $a,b$ be arbitrary elements in $\ggg$. Taking $b_1$ from $C(\ggg)\cap(b-C(\ggg))$ and letting $b_2=b-b_1$, we obtain
$b=b_1+b_2$ with $b_1,b_2\in C(\mathfrak{g})$. Using Case~2, we get
$$
\alpha(a+b)=\alpha(a+b_1+b_2)=\alpha(a+b_1)\oplus\alpha(b_2)=
$$
$$
=\alpha(a)\oplus\alpha(b_1)\oplus\alpha(b_2)=\alpha(a)\oplus\alpha(b_1+b_2)=\alpha(a)\oplus\alpha(b).
$$
\smallskip

This completes the proof of Theorem~\ref{TCC}.
\end{proof} 

If $Z(\mathfrak{R})\ne\{0\}$, then $\mathfrak{R}$ does not satisfy the C-condition. The following example proposed by Mikhail Zaitsev \cite[Remark~3.5]{A4} demonstrates that the converse statement does not hold.  

\begin{example}
\label{exx}
Consider a matrix Lie algebra

\smallskip

$$
\ggg=\left\{\left(
 \begin{array}{ccccc}
 a & 0 & a_{13} & a_{14} & a_{15} \\
 0 & a & a_{23} & a_{24} & a_{25} \\
 0 & 0 & a & 0 & a_{35} \\
 0 & 0 & 0 & a & a_{45} \\
 0 & 0 & 0 & 0 & 0 \\
 \end{array}
 \right);  \ a,a_{ij} \in K \right\}.
$$

\medskip

We claim that the center of this 9-dimensional Lie algebra is trivial, but the C-condition does not hold. In order to check this, we consider the commutator of two elements from $\ggg$:
$$
[A,B]=\left(
\begin{array}{ccccc}
 0 & 0 & 0 & 0 & c_{15} \\
 0 & 0 & 0 & 0 & c_{25} \\
 0 & 0 & 0 & 0 & c_{35} \\ 
 0 & 0 & 0 & 0 & c_{45} \\
 0 & 0 & 0 & 0 & 0 \\
\end{array}
\right),
$$
where
$$
c_{15}=ab_{15}+a_{13}b_{35}+a_{14}b_{45}-a_{15}b-a_{35}b_{13}-a_{45}b_{14}, \quad c_{35}= ab_{35}-a_{35}b
$$
and 
$$
c_{25}=ab_{25}+a_{23}b_{35}+a_{24}b_{45}-a_{25}b-a_{35}b_{23}-a_{45}b_{24}, \quad c_{45}= ab_{45}-a_{45}b.
$$
If $A$ is in the center of $\ggg$, then the condition $c_{15}=0$ for all $B$ implies 
$$
a=a_{13}=a_{14}=a_{15}=a_{35}=a_{45}=0,
$$ 
and the condition $c_{25}=0$ for all $B$ gives $a_{23}=a_{24}=a_{25}=0$, so $A=0$. 

\smallskip

Finally, the conditions $[A,D]=[B,D]=0$ are satisfied if we take $d=d_{35}=d_{45}=0$ and
$$
(d_{13},d_{14},d_{15})=(d_{23},d_{24},d_{25})=(u,v,w), 
$$
where $(u,v,w)$ is a non-zero solution of the system
\[
\left\{
\begin{aligned}
a_{35}x+a_{45}y-az & = 0\\
b_{35}x+b_{45}y-bz & = 0.\\
\end{aligned}
\right.
\] 
\end{example} 

\begin{problem}
Is the Lie algebra $\ggg$ from Example~\ref{exx} a UA-Lie ring?
\end{problem} 

%%%%%%%%%%%%%%%%%%%%%%%%%%%%%%%%%%%%%%%%%%%%%%%
\section{Semisimple Lie algebras and parabolic subalgebras}
\label{sla}

In this section we show that many important Lie algebras satisfy the C-condition. We assume that the group field $K$ is an algebraically closed field of characteristic zero. 

\begin{proposition} 
\label{PCC}
Let $\ggg$ be a semisimple Lie algebra over $K$ and $\ppp$ be a parabolic subalgebra in $\ggg$. Then the Lie algebra $\ppp$ satisfies the C-condition. 
\end{proposition} 

\begin{proof} 
Let $\bbb$ be a Borel subalgebra in $\ggg$ containing a Cartan subalgebra $\hhh$ and associated with a system of positive roots with respect to $\hhh$. We may assume that $\bbb$ is contained in $\ppp$.

Take an element $x\in\hhh$ that is not in the kernel of any root of $\ggg$ and let $y$ be the sum of root vectors over all positive roots of $\ggg$. Then $x,y\in\ppp$ and we have 
$C(x) = \hhh$ with $\hhh\cap C(y)=\{0\}$. This proves that the Lie algebra $\ppp$ satisfies the C-condition.
\end{proof}

Theorem~\ref{TCC} and Proposition~\ref{PCC} immediately imply the following result, first proved in~\cite[Theorem~1]{A3} and \cite[Corollary~4.6]{A4} for a weaker definition of unique addition. 

\begin{theorem}
\label{tmain}
Let $\ggg$ be a semisimple Lie algebra over an algebraically closed field $K$ of characteristic zero and $\ppp$ be a parabolic subalgebra in~$\ggg$. Then both $\ggg$ and $\ppp$ are UA-Lie rings. 
\end{theorem} 

Clearly, the arguments given above imply that if $\ggg$ is a direct sum of some classical Lie algebras $\sl_n(K)$, $\so_n(K)$ and $\sp_{2n}(K)$ over any field $K$ of characteristic zero, then 
$\ggg$ is a UA-Lie ring.  

Later a much more general result was proved by Mayorova. 

\begin{theorem} \cite[Theorem~6]{Ma}
Lie algebras of Chevalley type ($A_n$, $B_n$, $C_n$, $D_n$, $E_6$, $E_7$, $E_8$, $F_4$, and
$G_2$) over associative commutative rings with $\frac{1}{2}$ (with $\frac{1}{2}$ and $\frac{1}{3}$
in the case of $G_2$) are UA-Lie rings.
\end{theorem} 

The proof of this result is based on explicit computations with root systems.

%%%%%%%%%%%%%%%%%%%%%%%%%%%%%%%%%%%%%%%%%%%%%%
\section{Seaweed Lie algebras} 
\label{sw}

The aim of this section is to generalize the idea of the proof of Proposition~\ref{PCC} to a wider class of Lie algebras. Let $\aaa$ be a finite-dimensional Lie algebra over an infinite field $K$ 
and $\hhh\subseteq\aaa$ be a commutative subalgebra. We say that the pair $(\aaa,\hhh)$ is \emph{suitable} if $\aaa=\oplus_{i}\aaa_{\lambda_i}$, where $\lambda_i\in\hhh^*$ are linear functions on $\hhh$, $\aaa_{\lambda_i}:=\{x\in\aaa \, | \, [h,x]=\lambda_i(h)x \ \forall \, h\in\hhh\}$, we have $\cap_i\Ker\lambda_i=\{0\}$, and $\aaa_0=\hhh$. 

\begin{lemma}
\label{suit}
Let $(\aaa,\hhh)$ be a suitable pair. Then the Lie algebra $\aaa$ is a UA-Lie ring.
\end{lemma}

\begin{proof}
Take an element $h\in\hhh$ such that $\lambda_i(h)\ne 0$ for any non-zero $\lambda_i$. Then $C(h)=\hhh$.  If we let $y=\sum_i y_i$ with $0\ne y_i\in\aaa_{\lambda_i}$, then
$C(y)\cap\hhh=\{0\}$ and the algebra $\aaa$ satisfies the C-condition. So the assertion follows from Theorem~\ref{TCC}. 
\end{proof}

\medskip

Lemma~\ref{suit} can be applied in the following situation. Assume that the field $K$ is an algebraically closed field of characteristic zero, $\ggg$ is a semisimple Lie algebra over $K$ and 
$\hhh$ is a Cartan subalgebra in $\ggg$. A subalgebra $\aaa\subseteq\ggg$ is called \emph{regular} if $\aaa$ contains $\hhh$. In this case $\aaa$ as a vector space is a direct sum of $\hhh$
and some root subspaces $\ggg_{\alpha_i}$ in $\ggg$. 

We call a regular subalgebra $\aaa$ \emph{ample} if the roots $\alpha_i$ such that $\ggg_{\alpha_i}$ is contained in $\aaa$ span the vector space $\hhh^*$. Equivalently, this condition means that $\cap_i\Ker\alpha_i=\{0\}$.  Since there is no zero root, this condition is equivalent to the condition that the pair $(\aaa,\hhh)$ is suitable. 

\begin{example}
Any parabolic subalgebra $\ppp$ in a semisimple Lie algebra $\ggg$ is an ample regular subalgebra. 
\end{example}

\begin{proposition}
\label{reg}
A regular subalgebra $\aaa$ of a semisimple Lie algebra $\ggg$ is a UA-Lie ring if and only if $\aaa$ is ample. 
\end{proposition}

\begin{proof}
The ``if'' part follows from Lemma~\ref{suit}. As for the inverse implication, let $\hhh'=\cap_i\Ker\alpha_i$. Clearly, the subalgebra $\hhh'$ is contained in the center of $\aaa$. 

The derived subalgebra $[\aaa,\aaa]$ is spanned by all root subspaces $\ggg_{\alpha_i}$ and subspaces in $\hhh$ that are commutators of root subspaces $[\ggg_{\alpha_i},\ggg_{-\alpha_i}]$ corresponding to opposite roots. The latter subspaces lie in Lie subalgebras isomorphic to $\sl_2(K)$, and $\hhh\cap [\aaa,\aaa]$ is the Cartan subalgebra of the semisimple part of the algebra $\aaa$. It proves that $[\aaa,\aaa]\cap\hhh'=\{0\}$ and if $\hhh'\ne 0$, then $\aaa$ is not a UA-Lie ring by Proposition~\ref{neg}. 
\end{proof} 

A remarkable class of regular subalgebras of a semisimple Lie algebra $\ggg$ form so-called seaweed subalgebras introduced by Dergachev and Kirillov~\cite{DK} in the case $\ggg=\sl_n(K)$ and by Panyushev~\cite{Pa} in the general case; see also \cite{CCMR} for a nice recent survey. 

Recall that a subalgebra $\aaa$ in a semisimple Lie algebra $\ggg$ is called \emph{seaweed} if $\aaa$ is the intersection of two parabolic subalgebras $\ppp$ and $\ppp'$ such that $\ggg=\ppp+\ppp'$. Take a Borel subalgebra $\bbb\subseteq\ggg$ and a Cartan subalgebra $\hhh$ in $\bbb$. It is observed in~\cite{Pa} that up to conjugation one may assume that $\ppp$ contains $\bbb$ and $\ppp'$ contains $\bbb_{-}$, where $\bbb_{-}$ is the opposite Borel subalgebra. In particular, $\aaa$ is a regular subalgebra in $\ggg$.  

\smallskip 

Proposition~\ref{reg} immediatly implies the following result.

\begin{proposition}
\label{psw}
A seaweed subalgebra $\aaa$ of a semisimple Lie algebra $\ggg$ is a UA-Lie ring if and only if $\aaa$ is ample. 
\end{proposition}

%%%%%%%%%%%%%%%%%%%%%%%%%%%%%%%%%%%%%%%%%%%%%%
\section{Solvable Lie algebras and Mal'cev splittings} 
\label{ms}

In this section we show that the C-condition is fulfilled for a wide class of solvable Lie algebras. Further we consider finite-dimensional Lie algebras over an algebraically closed field $K$ of characteristic zero. 

\begin{definition}
A solvable Lie algebra $\rrr$ is {\it Mal'cev split} if $\rrr=\kkk\oplus\nnn$, where $\mathfrak{n}$ is a nilpotent ideal, $\mathfrak{k}$ is a commutative subalgebra, and the adjoint action of $\kkk$ on $\nnn$ is faithful and completele reducible. 
\end{definition} 

Here $\nnn=\oplus_i\nnn_{\lambda_i}$, where 
$$
\nnn_{\lambda_i}=\{x\in\nnn \, | \, [y,x]=\lambda_i(y)x \ \forall y\in\kkk\}
$$
and $\{\lambda_i\}$ are linear functions on $\kkk$ with $\cap_i\Ker\lambda_i=\{0\}$. 

\bigskip

Any algebraic solvable Lie algebra is Mal'cev split. Moreover, every solvable Lie algebra $\rrr$ admits the {\it Mal'cev splitting}, i.e. an embedding $\gamma\colon\rrr\to M(\rrr)$, where $M(\rrr)=\kkk\oplus\nnn$ is Mal'cev split and $\gamma(\rrr)$ is an ideal in $M(\rrr)$ with $M(\rrr)=\kkk+\gamma(\rrr)=\gamma(\rrr)+\nnn$. The Mal'cev splitting of $\rrr$ exists and is unique up to isomorphism; see~\cite{Mal}. 

\begin{definition}
A solvable Lie algebra $\rrr$ is {\it admissible} if for $M(\rrr)=\kkk\oplus\nnn$ there is $x\in\nnn$ with $C(x)\cap\nnn_0=\{0\}$. 
\end{definition} 

\begin{example}
Any parabolic subalgebra $\ppp$ in a semisimple Lie algebra $\ggg$ is a Lie algebra with admissible radical. 
\end{example}

The following result is obtained in~\cite{AT} in a slightly weaker form. The rest of this section is devoted to its proof.

\begin{theorem}
\label{tar}
Let $\ggg$ be a finite-dimensional Lie algebra with admissible radical. Then $\ggg$ satisfies the C-condition. In particular, $\ggg$ is a UA-Lie ring. 
\end{theorem}

\begin{proof}
Let $\ggg=\lll\rightthreetimes\rrr$, where $\lll$ is a maximal semisimple subalgebra in $\ggg$ and $\rrr$ is the radical of $\ggg$. We know from Proposition~\ref{PCC} that $\lll$ satisfies the C-condition. 

\begin{lemma}
\label{ll1}
The Lie algebra $M(\rrr)$ satisfies the C-condition. 
\end{lemma}

\begin{proof}
Take $a\in\kkk$ with $\lambda_i(a)\ne 0$ for all non-zero $\lambda_i$. Then $C(a)=\kkk\oplus\nnn_0$. By definition of admissibility, there exists an element $x\in\nnn$ such that 
$C(x)\cap\nnn_0=\{0\}$. This condition holds for all elements $x$ from a Zariski open dense subset $V$ in $\nnn$, because if the rank of the linear map $\nnn_0\to\nnn$, $y\mapsto [x,y]$ 
is maximal at some point $x\in\nnn$, it is maximal for all points $x$ from an open dense subset. Take $b=\sum b_{\lambda_i}$ with $0\ne b_{\lambda_i}\in\nnn_{\lambda_i}$ and
$C(b)\cap\nnn_0=\{0\}$. 

We claim that $C(b)\cap(\kkk\oplus\nnn_0)=\{0\}$. Indeed, assume that $d+e\in C(b)$, $d\in\kkk$, $e\in\nnn_0$. Then $[d+e,b]=0$. Since $[\nnn_0,\nnn_{\lambda_i}]\subseteq\nnn_{\lambda_i}$, we have $\lambda_i(d)b_{\lambda_i}=-[e,b_{\lambda_i}]$. But the operator $\ad(e)$ is nilpotent, so all $\lambda_i(d)=0$. The condition $\cap_i\Ker\lambda_i=\{0\}$ implies $d=0$, and the condition $C(b)\cap\nnn_0=\{0\}$ shows that $e=0$. 

We conclude that $C(a)\cap C(b)=\{0\}$. 
\end{proof}

\begin{lemma}
The Lie algebra $\rrr$ satisfies the C-condition. 
\end{lemma}

\begin{proof}
Take elements $a,b\in M(\rrr)$ as in Lemma~\ref{ll1}. By definition of the Mal'cev splitting, there are elements $p_1,p_2\in\rrr$ such that $a=p_1+q$ and $b=r+ p_2$ with $q\in\nnn$
and $r\in\kkk$. 

We can assume that $q\in\nnn_0$. Indeed, let $q=\sum_i q_{\lambda_i}$. Then $[a,q]=\sum_i\lambda_i(a)q_{\lambda_i}=[p_1,q]\in\rrr$ and $\lambda_i(a)\ne 0$ for all non-zero $\lambda_i$. Adding to $q$ the element $[a,q]\in\rrr\cap\nnn$ with a suitable coefficient, we cancel one of the components $q_{\lambda_i}$ with non-zero $\lambda_i$. Continuing this process, we come to $q\in\nnn_0$.

We claim that $C(p_1)\subseteq\kkk\oplus\nnn_0$. Indeed, if $[a-q,l]=0$, then $\sum_i\lambda_i(a)l_{\lambda_i}=[q,l]$. But the operator $\ad(q)$ is nilpotent and so all $\lambda_i(a)l_{\lambda_i}=0$. For non-zero $\lambda_i$ we have $\lambda_i(a)\ne 0$, so $l_{\lambda_i}=0$. 

Finally, let us show that $C(p_2)\cap(\kkk\oplus\nnn_0)=\{0\}$. If $[b-r,t]=[b,t]=0$ with $t\in\kkk\oplus\nnn_0$, then $t\in C(b)$. But it is shown in the proof of Lemma~\ref{ll1} that
$C(b)\cap(\kkk\oplus\nnn_0)=\{0\}$.

Thus, we proved that $C(p_1)\cap C(p_2)=\{0\}$, 
\end{proof}

\begin{lemma}
\label{lemmel}
If a finite-dimensional Lie algebra $\fff$ satifies the C-condition, then there is a Zariski open dense subset $U\subseteq\fff\oplus\fff$ such that $C(a)\cap C(b)=\{0\}$ for all $(a,b)\in U$.
\end{lemma}

\begin{proof}
If for some pair $(a,b)$ the rank of the linear map $\fff\to\fff\oplus\fff$, $c\mapsto ([a,c],[b,c])$ equals $\dim\fff$ and, in particular, is maximal possible, then the same holds for all pairs
$(a,b)$ from a Zariski open dense subset $U \subseteq\fff\oplus\fff$. 
\end{proof}

Now we are going to check the C-condition for the Lie algebra $\ggg$. Take Zariski open dense subsets $U_1\subseteq\lll\oplus\lll$ and $U_2\subseteq\rrr\oplus\rrr$ as in Lemma~\ref{lemmel}. In particular, if $(z_1,z_2)\in U_2$ then $C(z_1)\cap C(z_2)\cap\rrr=\{0\}$. It follows that we have $C(x_1)\cap C(x_2)\cap\rrr=\{0\}$ for all $(x_1,x_2)$ from 
a Zariski open dense subset $W\subseteq\ggg\oplus\ggg$. Indeed, we can apply the same arguments as in the proof of Lemma~\ref{lemmel} to the linear map $\rrr\to\ggg\oplus\ggg$,
$z\mapsto ([x_1,z],[x_2,z])$.

Choose elements $x_1=y_1+z_1, x_2=y_2+z_2\in\ggg$ with $(y_1,y_2)\in U_1$, $z_1,z_2\in\rrr$, and $(x_1,x_2)\in W$. This is possible because two Zariski open dense subsets $U_1\oplus\rrr\oplus\rrr$ and $W$ in $\ggg\oplus\ggg$ have non-empty intersection. 

If $x=y+z\in C(x_1)\cap C(x_2)$, then $y\in C(y_1)\cap C(y_2)$, so $y=0$. It implies $[z,x_1]=[z,x_2]=0$, and $z=0$ by assumption on the subset $W$. It means that $C(x_1)\cap C(x_2)=\{0\}$, and Theorem~\ref{tar} is proved.
\end{proof}

%%%%%%%%%%%%%%%%%%%%%%%%%%%%%%%%%%%%%%%%%%%%%%
\section{UA-Lie algebras}
\label{uala}

Now we try to take into account the specifics of Lie algebras in the definition of the uniqueness of addition.

\begin{definition}
Let $K$ be a field. Consider two $K$-vector spaces $V$ and $W$. A map $\alpha\colon V\to W$ is {\it semilinear} if it is additive, i.e. $\alpha(v_1+v_2)=\alpha(v_1)+\alpha(v_2)$ for all $v_1,v_2\in V$, and there is an automorphism $\theta\in\text{Aut}(K)$ such that $\alpha(\lambda v)=\theta(\lambda)\alpha(v)$ for all $\lambda\in K$ and $v\in V$. 
\end{definition}

\begin{remark}
Some fields like $\mathbb{Q}$ or $\mathbb{R}$ have no nontrivial automorphism, and over such fields any semilinear map is linear. 
\end{remark}

\begin{definition}
A Lie algebra $(\ggg,+,[,])$ over a field $K$ is called a {\it unique addition Lie algebra}, or a {\it UA-Lie algebra}, if for any Lie algebra $(\sss,\oplus,\lceil,\rceil)$ over the same field $K$ any bijection $\alpha\colon\ggg\to\sss$ preserving commutators is a semilinear map. 
\end{definition} 

We also have a weaker version of the unique addition property. 

\begin{definition}
\label{dbdb}
A Lie algebra $\ggg$ is a {\it Lie algebra with weak unique addition}, or a {\it wUA-Lie algebra}, if any bijection $\alpha\colon\ggg\to\ggg$ preserving commutators is a semilinear map.  
\end{definition} 

\begin{remark}
We have no example of two Lie algebras $\ggg_1$ and $\ggg_2$ over a field $K$ such that there exists a semilinear bijection $\ggg_1\to\ggg_2$ preserving
commutators, but the Lie algebras $\ggg_1$ and $\ggg_2$ are not isomorphic. 
\end{remark}

Let us denote by $\sss_2(K)$ the 2-dimensional non-commutative Lie algebra over $K$. The following result is proved in~\cite{A1,A2}.

\begin{proposition}
Let $K$ be a field of characteristic different from $2$. Then the Lie algebras $\sss_2(K)$ and $\sl_2(K)$ are wUA-Lie algebras.  
\end{proposition} 

Let us assume that $K$ has characteristic zero. Clearly, the Lie algebras $\sss_2(K)$ and $\sl_2(K)$ satisfy the C-condition, so they are UA-Lie rings. Over the field $\QQ$ of rational numbers any additive map between vector spaces is linear, so we conclude that $\sss_2(\QQ)$ and $\sl_2(\QQ)$, as well as any other Lie algebra over $\QQ$ that is a UA-Lie ring, are UA-Lie algebras. 

\smallskip

Let us give one more observation on this subject.

\begin{lemma}
Let $K$ be an algebraically closed field of characteristic zero. Assume that $\sss$ is a finite-dimensional Lie algebra over $K$. Then any bijection $\alpha\colon\sl_2(K)\to\sss$ preserving commutators is a semilinear map. 
\end{lemma}

\begin{proof}
The Lie algebra $\sss$ inherits properties of $\sl_2(K)$ that are expressible in terms of commutators. For instance, it is not solvable and the centralizer of any non-zero element in $\sss$ is commutative. By~\cite[Proposition~2]{AMP}, these two properties imply that $\sss$ is isomorphic to~$\sl_2(K)$. So, the UA-property for $\sl_2(K)$ follows from the wUA-property.    
\end{proof} 

Summarizing this discussion, it is worth noting that very few examples of UA-Lie algebra are known, and building such examples seems to be an important and interesting task.

\smallskip

Finally, let us mention one more fact related to this subject. In~\cite{Pon2}, Ponomarev proved the following result.
 
\begin{theorem}
Let $K$ be an algebraically closed field of characteristic zero. Then a parabolic subalgebra $\ppp$ in a semisimple Lie algebra $\ggg$ over $K$ is a wUA-Lie algebra if and only if the Lie algebra 
$\ggg$ is simple.
\end{theorem} 

The proof uses Theorem~\ref{tmain} and is based on a description of automorphisms of regular subalgebras of reductive Lie algebras as Lie rings obtained in~\cite{Pon1}. 

%%%%%%%%%%%%%%%%%%%%%%%%%%%%%%%%%%%%%%%%%%%%%%
\section{Commutator-preserving injections} 
\label{cpi}

One may weaken Definition~\ref{deflie} and assume that the map $\alpha$ is not a bijection, but just an injection. This weakening changes the situation significantly. 

\begin{lemma}
\label{pol}
Assume that the field $K$ contains at least $3$ elements and $\ggg$ is a Lie algebra over $K$ with $\ggg\ne [\ggg,\ggg]$. Then there is a Lie algebra $\sss$ over $K$ and an injective
map $\beta\colon\ggg\to\sss$ that preserves commutators, but is not additive. 
\end{lemma}

\begin{proof}
Let us fix a non-additive map $\gamma\colon K\to K$ with $\gamma(0)=0$. For example, we may let $\gamma(a)=1$ for any non-zero $a\in K$. 

Take a basis in $[\ggg,\ggg]$ and complement it to a basis $\{x_i\}$ in $\ggg$. We assume that $x_1\notin [\ggg,\ggg]$. Let the Lie algebra $\sss$ be the direct sum of Lie algebras
$\ggg\oplus\langle z\rangle$, so $z$ is a central element. Define 
$$
\beta\colon\ggg\to\sss, \quad \beta(x)=x+\gamma(a_1)z, \quad \text{where} \quad x=\sum_i a_ix_i, \quad a_i\in K.
$$
Then we have
$$
\beta([x,y])=[x,y]+\gamma(0)z=([x,y],0)=[\beta(x),\beta(y)] \quad \text{for all} \quad x,y\in\ggg.
$$
So the map $\beta$ preserves commutators, but it is not additive. 
\end{proof}

Since the Lie algebra $\sss_2(K)$ satisfies the C-condition, it is a UA-Lie ring for any infinite field $K$. 
At the same time, Lemma~\ref{pol} shows that a commutator-preserving injection need not be additive in this case. 

\smallskip

In view of the above, let us formulate the following conjecture. 

\begin{conjecture}
Let $\beta\colon\ggg\to\SSS$ be a commutator-preserving injection from a semisimple Lie algebra $\ggg$ to a Lie ring $\SSS$. Then the map $\beta$ is additive.  
\end{conjecture}

At the moment we have no example of a Lie ring $\RRR$ such that any injective map $\beta\colon\RRR\to\SSS$ to a Lie ring $\SSS$ that preserves commutators is automatically additive. 

%%%%%%%%%%%%%%%%%%%%%%%%%%%%%%%%%%%%%%%%%%%%%%

\end{document}